\documentclass[a4paper,12pt]{article}
\usepackage{hyperref}
\usepackage{amsmath}
\usepackage{amssymb}
\usepackage{amsthm}
\usepackage{mathrsfs}
\usepackage{todonotes}
\usepackage{tikz}
\usepackage{enumitem}
\usetikzlibrary{arrows}
\usetikzlibrary{shapes}
\usepackage{subcaption}
\usepackage[backend=biber,style=numeric, maxbibnames=99]{biblatex}
\bibliography{RecombinationILSMv2}
\newtheorem{theorem}{Theorem}[section]
\newtheorem{lemma}[theorem]{Lemma}
\newtheorem{proposition}[theorem]{Proposition}
\newtheorem{definition}[theorem]{Definition}

\newtheorem{remark}[theorem]{Remark}

\newcommand{\EE}{\mathbb{E}}

\newcommand{\NN}{\mathbb{N}}

\newcommand{\PP}{\mathbb{P}}

\newcommand{\GG}{\mathbb{G}}
\newcommand{\RR}{\mathbb{R}}
\newcommand{\AS}{\mathbb{S}}

\newcommand{\A}{{\cal A}}
\newcommand{\B}{{\cal B}}

\newcommand{\F}{{\cal F}}
\newcommand{\G}{{\cal G}}

\newcommand{\aL}{{\cal L}}

\newcommand{\V}{{\cal V}}

\newcommand{\Y}{{\cal Y}}

\newcommand{\ds}{\displaystyle}
\newcommand{\lt}{\leadsto}
\begin{document}

\title{\itshape A probabilistic analysis of a continuous-time 
evolution in recombination.}
\author{Ian Letter and Servet Mart\'inez}

\maketitle

\begin{abstract}
We study the continuous-time evolution of the recombination equation of population genetics. This evolution is given by a differential equation that acts on a product probability space, and its solution can be described by a Markov chain on a set of partitions that converges to the finest partition. We study an explicit form of the law of this process by using a family of trees. We also describe the geometric decay rate to the finest partition and the quasi-limiting behaviour of the Markov chain when conditioned on the event that the chain does not hit the limit.
\end{abstract}

\medskip


\medskip


\section{Introduction.} 
\label{intro}
Let $I$ be a finite set of sites, $(\A_i, \B_i)_{i \in I}$ be a family of measurable spaces, where $\B_i$ is the corresponding $\sigma$-algebra on $\A_i$ and $\mathcal{P}_I$ be the set of probability measures on the product measurable space $(\prod_{i\in I} \A_i, \otimes_{i \in I} \B_i)$. Here we study the evolution of the following ordinary differential equation, acting on $\mathcal{P}_I$:
\begin{equation} \label{introe} \begin{cases} \dot{\omega}_t = \ds \sum_{\delta \in \G} \rho_\delta ( R_\delta - 1) \omega_t, \\  \omega_0 = \mu. \end{cases} \end{equation}
Here $\G$ is a set of partitions of the set $I$, $\rho = (\rho_\delta)_{\delta \in \G}$ is a set of rates, which are non-negative numbers, $R_\delta(\mu) := \ds \otimes_{L \in \delta} \mu_L$ is the
product measure, and $\mu_J$ is the marginal of $\mu$ on $(\prod_{i\in J} A_i, \otimes_{i \in J} \B_i)$.

\medskip

This evolution equation is used to model an infinite population under the action of genetic recombination, where the recombination of genes with the partition $\delta$ happens at rate $\rho_\delta$. More precisely, with rate $\rho_\delta$ for $\delta = \{a_1, . . . , a_r\}$, a new individual is formed by inheriting the sites in $a_i$ from parent $i$, for $1 \leq i \leq r$. We assume each parent is chosen at random on the population, and after this process one of the parents is killed at random, chosen uniformly between the parents. This interpretation can be explained from two points of view. First, the equation can be seen as a mass balance equation: For every $\delta \in \G_\rho$, an individual with genetic sequence given by $x$ is produced from the corresponding parental sequences at overall rate $\rho_\delta R_\delta(\omega)$, where $R_\delta(\omega)$ reflects the random combination on the population given by the random selection of the parents; and at the same time, individuals with genetic sequence of type $x$ are lost (i.e., replaced by new ones) at overall rate $\rho_\delta \omega_t(x)$, since one of the parents is lost to form the new individual. The second point of view was given in \cite{duality}, where it is shown that the equation can be seen as a limit of a finite population model. Loosely speaking, one can consider $N$ individuals such that any of them suffers recombination of genes with the partition $\delta$  at rate $\rho_\delta$. Let $Z^N_t$ be the counting random measure of the genetic types on $( \prod_{i \in I} A_i, \otimes_{i \in I} \B_i$), that is, it counts how many individuals of a certain type are at time $t$. It has been proven in \cite{duality} that $Z_t^N/N$ converges, when $N$ goes infinity, to the solution of equation (\ref{EDOREC}) (this convergence is in probability uniformly on compacts sets). Further discussions about the equation and the model can be found in \cite{lyu}, \cite{cl1},\cite{cl2}, \cite{bbs}. 

\medskip

Given the above explanation is that equation (\ref{introe}) has served to model the evolution of the genetic composition of a population under recombination. At an individual level, recombination is the genetic mechanism in which two parental individuals create an offspring with sexual reproduction, mixing their genetic components. During recombination, crossover events happen between parents, that is, their genetic material is cut into two parts and then it is exchanged, which produces genes with parts coming from both parents. Note that multiple crossovers could happen on one single recombination event. The way in which the crossover process and parent selection are modelled can lead to different equations having different shapes, in special being stochastic or deterministic. In particular, equation (\ref{introe}) employs the deterministic continuous-time approach, but allow general crossover patterns including more than two parents. Further details on the variety of approaches to recombination can be found in \cite{bbs} or \cite{duality}.

\medskip

We mention some historical work related to the recombination equation. One of the first descriptions of this process dates back to Morgan \cite{morgan} in 1911. The first solution in discrete time in a simpler space is due to due to Geiringer in 1944 \cite{hg}. Later, Lyubich in 1992, analyzed the structure of the solutions to this problem and explored some connections with stochastic processes \cite{lyu}. A couple of years later, Christiansen in 1999 and B\"{u}rger in 2000, explored generalization of this analysis. In particular on more complex structures (due to Christiansen \cite{cl2}) and adding other effects to the population dynamics, like selection and mutation (due to B\"{u}rger \cite{cl1}). 

\medskip

In spite of all these results, the dynamics of the population under recombination continue attracting a lot of attention nowadays. Recent literature is devoted to the analysis of the differential equation (\ref{introe}) and its analogue in discrete time. We mention some work that motivates our study. The evolution equation was introduced in the measure theory framework on \cite{bbfinal} and a recursive solution was given. In \cite{bbs} the relationship between the solution of the equation and a Markov chain (the fragmentation process) is established. In \cite{duality} it is stated a duality relationship between this stochastic process and the deterministic formulation of the crossover patterns. Finally, in \cite{sm} the quasi-limiting distribution is given for the process in discrete time, giving a great understanding of its long time behaviour. For more details on the model, we refer to the introductory section of any of the works aforementioned.

\medskip

This work focus on answering questions arising in those studies: to give an explicit solution to the dynamics in (\ref{introe}), explore a solution through a Markov chain and then study the quasi-limiting behaviour of the associated Markov chain. 

\medskip

As said, the equation (\ref{intro}) was studied in a general framework in \cite{bbs}, and its solution was given by a recursion formula using tools from combinatorics and differential equation. This solution was studied via a Markov fragmentation process in the particular case of single-crossover in \cite{prune} and \cite{bvw} by combining techniques from probability and graph theory. It remains open if the general case can be also studied using this kind of tools to obtain an explicit solution instead of a recursive one. We give a positive answer to this question and we supply an explicit solution to the equation (\ref{intro}) in terms of a family of graphs we call the fragmentation trees. In relation to the Markov fragmentation process, we obtain the quasi-limiting behaviour, that is, its behaviour when avoiding the limiting state. This is done by using similar techniques as those used in \cite{sm} for the discrete-time model.

\medskip

Our work is organized as follows. In Section 2 we give the required notation behind the recombination equation, mainly partitions and measures on probability spaces. In Section 3 we formulate the equation and relate to it a continuous time Markov process called the fragmentation process. Our results are shown in Sections 4 and 5; the main ones being Theorem \ref{LAWF} and Theorem \ref{QUASI1}. In the first one, we give a formula for the law of the fragmentation process in terms of a family of trees. In the second result, we characterize the quasi-limiting behaviour of the process before attaining its absorbing state. We emphasize that the main interest in quasi-limiting behaviour is that this gives very precise information on the deviations of the Markov chain from the limit state. This result also allows obtaining a formula for approximating the solution of the recombination equation. This is given in Theorem \ref{aproximation}. For applications, these two results are interesting for the computability of solutions of equation (\ref{introe}). Theorem \ref{LAWF} gives an explicit formula for the solutions of equation (\ref{introe}) in terms of a finite set. This contrast with the classic solutions given by the semi-group with depends on an infinite sum. Theorem \ref{aproximation} on the other hand provides a fine approximation for the solution of equation (\ref{introe}) in the case where the finite set mentioned in Theorem \ref{LAWF} is too big for allowing an explicit computation. 

\section{Partitions} \label{part}

Let $I$ be a nonempty finite set. A partition $\delta=\{L: L\in \delta\}$ of $I$ is a collection of nonempty and pairwise disjoint sets that cover $I$, any of the sets $L$ belonging to $\delta$ is called an atom of $\delta$. We note by $\AS(I)$ the family of 
partitions of $I$.

\medskip

For $\delta, \delta'\in \AS(I)$, $\delta'$ is said to be finer than $\delta$ or $\delta$ is coarser than $\delta'$, we note $\delta\preceq \delta'$, if every atom of $\delta'$ is contained in an atom of $\delta$. This is an order relation. The finest partition is $\{\{i\}: i\in I\}$, and the coarsest one is the trivial partition $\{I\}$ having a single  atom. For $\delta, \delta'\in \AS(I)$, we use $\delta\vee \delta'$ for the common refinement between the two partitions, that is
\[ \delta \vee \delta' := \{L\cap L': L\in \delta, L'\in \delta'\} \setminus \{ \emptyset \}.\] The operation $\vee$ is commutative, associative and $\{I\}$ is its unit element because $\{I\}\vee \delta=\delta$ for all $\delta\in \AS(I)$. One has $\delta\preceq \delta'$ if and only if $\delta \vee \delta'=\delta'$. 

\medskip
If $\delta$ is a partition and $J \subseteq I$ is a nonempty subset we note $\delta|_J = \{ L \cap J: L \in \delta\}$ the partition induced by $\delta$ on $J$. So, for $\delta' \in \AS(J^c)$ we have $\delta|_J \cup \delta' \in \AS(I)$.
\medskip

Let us fix $\G$ a nonempty family of partitions of $I$. Let $\delta \in \AS(I)$, $a \in \delta$, $\gamma \in \G$, we denote $\delta \lt_a^\gamma \delta'$ if $(\delta \setminus \{ a \}) \cup \gamma|_a = \delta'$. By definition the atom $a \in \delta$ is unique, so it can be noted $a(\delta,\delta')$, but there could exist several $\gamma \in \G$ fulfilling the condition. We also put $\delta \lt \delta'$ when $\delta \lt_a^\gamma \delta'$ for some $a \in \delta, \gamma \in \G$, and we say $\delta'$ is a fragmentation of $\delta$. 

\medskip

Now we associate to $\G$ the following sequence of families of partitions, which are the consecutive fragmentations of $\G$:
\begin{align}
\label{eqg1}
\Y_0(\G) &= \{ \{ I \} \}, \nonumber \\
\forall\, n\ge 1: \quad
\Y_{n+1}(\G)&=\{(\delta \setminus \{a\}) \cup \gamma|_a: \delta\in 
\Y_{n}(\G), a \in \delta, \gamma \in \G\}. 
\end{align}
It can be easily checked that for all $n \geq 1, \delta\in \Y_n(\G)$ satisfies $\delta \lt \delta$ so $\Y_n(\G)\subseteq \Y_{n+1}(\G)$ for all $n\ge 1$. This sequence stabilizes in a finite number of steps, that is there exists $n_0\ge 1$ such that $\Y_{n_0+k}(\G)=\Y_{n_0}(\G)$ for all $k\ge0$. Let
\begin{equation}
\label{eqg2}
\Y^*(\G)=\bigcup_{n\ge 0} \Y_n(\G).
\end{equation}
Note that $\Y^*(\G) = \Y_{n_0}(\G) \cup \{ \{I\} \}$. Denote by $\gamma^\G$ the partition which is the common refinement of all the partitions in $\G$, meaning that
$$
\gamma^\G=\bigvee_{\gamma \in \G} \gamma.
$$
This is the finest partition in $\Y^*(\G)$, that is $\delta\preceq \gamma^\G$ for all $\delta\in \Y^*(\G)$. The atoms of $\gamma^\G$ are the nonempty intersections $\bigcap_{\gamma \in \G} L_\gamma$, where $(L_\gamma: \gamma \in \G)$ varies over all the sequences of atoms of the partitions in $\G$. 

\medskip

The partition $\gamma^\G$ is the unique element in $\Y^*(\G)$ that satisfies $\gamma^\G \vee \gamma=\gamma^\G$ for all $\gamma \in \G$. Also $\gamma^\G\vee \delta=\gamma^\G$ for all $\delta\in \Y^*(\G)$, which means that $\gamma^\G$ is an absorbing element in $(\Y^*(\G),\vee)$. 

\begin{remark}
\label{rem3a}
If one redefines $I$ as the set of atoms of the partition $\gamma^\G$ one can always assume that the atoms of $\gamma^\G$ are singletons, that is $\gamma^\G=\{\{i\}: i\in I\}$. We will not do it because there is no substantial gain in notation.
\end{remark}

\subsection{Product probability spaces}
Let $(\A_i,{\cal B}_i)_{i\in I}$ be a finite collection of measurable spaces and let $\prod_{i\in I}\A_i$ be a product space endowed with the product $\sigma-$field $\otimes_{i\in I} {\cal B}_i$. Denote by ${\cal P}_I$ the set of probability measures on $(\prod_{i\in I}\A_i,\otimes_{i\in I} {\cal B}_i)$. Let $J\subseteq I$ and ${\cal P}_J$ be the set of probability measures on $(\prod_{i\in J}\A_i,\otimes_{i\in J} {\cal B}_i)$. The marginal $\mu_J\in {\cal P}_J$ of $\mu\in {\cal P}_I$ on $J$, is given by
\[ \forall C\in \otimes_{i\in J} {\cal B}_i:\quad \mu_J(C)=\mu(C\times \prod_{i\in J^c}\A_i).\]
For $J=I$ we have $\mu_I=\mu$, and we put $\mu_\emptyset\equiv 1$ to get consistency in all the relations where it will appear, in particular in product measures.

\medskip

Let $J,K\subseteq I$, $J\cap K=\emptyset$. For $\mu_J\in {\cal P}_J$, $\mu_K\in {\cal P}_K$, we denote by $\mu_J\otimes \mu_K$ its product measure. We have that $\otimes$ is commutative and associative, $\mu_\emptyset=1$ is the unit element, and $\otimes$ is stable under  restriction, that is, for all $J, K, M\subseteq I$ with $J\cap K=\emptyset$ and $M\subseteq J\cup K$,
\begin{equation}
\label{eab}
(\mu_J\otimes \mu_{K})_M=\mu_{J\cap M}\otimes \mu_{K\cap M}.
\end{equation}
Associated to $\otimes$ we define the recombination of a measure $\mu \in \mathcal{P}_I$ by a partition $\delta \in \AS(I)$ by:
\[ R_\delta (\mu) = \bigotimes_{L \in \delta } \mu_L. \]
As seen in \cite{bbs} this operator is Lipschitz of constant $2|\delta|+1$ with respect to the norm of total variation $|| \cdot ||$. We recall that, for $\mu, \nu \in \mathcal{P}_I$, $|| \mu - \nu || = (\mu-\nu)_+(\prod_{i \in I} A_i) + (\mu-\nu)_- (\prod_{i \in I} A_i)$, where $(\mu-\nu)_+$ and $(\mu-\nu)_-$ are the (non-negative) measures called the positive part and the negative part respectively.

\section{The equation and its solution}

Let $(\rho_\delta)_{\delta \in \AS(I)}$ be a collection of non-negative real numbers, called the recombination rates. Let $\G_\rho$ be the support of $\rho$, that is:
\[ \G_\rho := \{ \delta : \rho_\delta > 0 \}. \]
We note by $|\rho| = \sum_{\gamma \in \G_\rho} \rho_\gamma$ the total mass of $\rho$. We assume that $|\rho|>0$ so that $\G_\rho \not = \emptyset$.
We are interested in studying the following ordinary differential equation, which acts on $(\mathcal{P}_I, || \cdot ||)$ (see \cite{bbs}):
\begin{equation} \label{EDOREC} \begin{cases} \frac{d \omega_t}{d t} = \ds \sum_{\delta \in \G_\rho} \rho_\delta ( R_\delta - 1) \omega_t \: , \\ \omega_0 = \mu \:, \end{cases} \end{equation}
here $1 \omega = \omega$. Since $R_{\{ I \}}(\mu) = \mu$ we will assume $\rho_{\{ I \}} =0$. Given that $R_\delta$ is Lipschitz it follows that problem (\ref{EDOREC}) admits a unique solution. We will denote by $\Xi_t \mu$ the solution, at time $t$, with initial condition $\mu$. We will find an expression for the solution of the equation in terms of a Markov chain.

\medskip

\begin{definition} \label{fragprosdef}
The \textit{fragmentation process} is the continuous time Markov process $(X_t)_{t \geq 0}$ taking values on $\Y^*(\G_\rho)$ whose Markov generator is given by: 
\[\forall \delta,\delta' \in \Y^*(\G_\rho): \: \begin{cases}
    Q_{\delta, \delta'}  =  \ds \sum_{ \substack{\gamma \in \G_\rho: \\ \delta \lt_a^{\gamma} \delta'}} \rho_{\gamma}  &  \text{if} \: \delta \lt \delta'; \\
    Q_{\delta, \delta'} = 0  &  \text{if} \: \delta' \not = \delta,\delta \not \lt \delta';\\
    Q_{\delta, \delta} = \ds  - \sum_{\substack{\delta' \in \Y^*(\G): \\ \delta \lt \delta', \delta \not = \delta' }} Q_{\delta, \delta'}. 
\end{cases}\]
\end{definition}
Recall that if $\delta \lt \delta'$ then $a=a(\delta,\delta')$ is uniquely defined, so there is no ambiguity in the definition of $Q_{\delta,\delta'}$ when $\delta \lt \delta'$. The definition of $Q_{\delta,\delta}$ ensures we are in the conservative case. 

\medskip

For $\delta \in \mathcal{Y}^*(\mathcal{G}_\rho)$, we denote by $\PP_{\delta}$ the law of the process starting on the state $\delta$, and by $\EE_\delta$ the associated expected value. As a abuse of notation we will use $\PP$, $\EE$ for $\PP_{\{ I \}}$ and $\EE_{\{I \} }$ respectively. It can be checked that $Q_{\{ I \},\{ I \} } = -|\rho|$. For $\delta,\delta' \in \Y^*(\G_\rho), t \geq 0$ we use $P_{\delta,\delta'}(t) = \PP_\delta(X_t=\delta')$ to note the transition semigroup.
\begin{remark} \label{rem:ia}
The process satisfies the following property: if $X_t = \{ A_1 ,.., A_r \}$ then each atom $A_i$ splits up into a partition $B_i$ at rate $\ds \sum_{\substack{ \gamma \in \G_\rho: \\ \gamma|_{A_i} = B_i}} \rho_\gamma$, independently on each atom. When the process makes a jumps it evolves to a finer partition, and so if it exits from a state it does never returns to it.
\end{remark}

\begin{remark} \label{fragmeaning}
The fragmentation process can be seen as the action of recombination on the ancestry of the genetic material of an individual backwards in
time. Namely, if a sequence is pieced together according to a partition $\delta = \{a_1, . . . ,a_r\}$ from
various parents forwards in time, then the sequence is partitioned into the parts
of $\delta$ when we look backwards in time, where each part $a_i$ is associated with a different parent.
It can be seen that the fragmentation process is the limit of the analogue stochastic process in finite populations, see \cite{duality}, \cite{bbs}, \cite{Durrett}.
\end{remark}

The next theorem is a continuous time version of the discrete time analogue in \cite{sm}, Proposition 3.2. Even though the framework is different this result was proved in \cite{bbfinal}, Theorem 2. Hence, we omit the proof.
\begin{theorem} \label{RECSOL}
Let $(X_t)_{t \geq 0}$ be the fragmentation process, and let $\mu^{(t)} := R_{X_t} (\mu) = \bigotimes_{L \in X_t} \mu_L$ be the recombination of $\mu$ by $X_t$.  Then:
\[ \Xi_t \mu  = \mathbb{E}(\mu^{(t)}) := \sum_{\delta \in \Y^*(\G_\rho)} \PP(X_t = \delta) \bigotimes_{L \in \delta} \mu_L. \]
\end{theorem}

Theorem \ref{RECSOL} transform the analysis of equation (\ref{EDOREC}) in the analysis of the quantity $\PP(X_t = \delta)$. In particular obtaining explicit expression or approximations for $\PP(X_t = \delta)$ turn into explicit expression for  (\ref{EDOREC}). With this in mind our objective in the next Section is give a explicit formula for $\PP(X_t = \delta)$, given by Theorem \ref{LAWF}. This implies an explicit formula for $\Xi_t \mu$. In Section \ref{Sec:lim} we study the asymptotic behaviour of $(X_t)_{t \geq 0}$, which translate in approximations of $\PP(X_t = \delta)$. By Theorem \ref{RECSOL} this gives an approximation of $\Xi_t \mu$, which is Theorem \ref{aproximation}

\medskip

\section{Law of the fragmentation process}

\subsection{Fragmentation trees}

Our objective is to obtain a formula for the law of the fragmentation process, that is, to be able to compute $\PP(X_t = \delta)$ for any $\delta \in \Y^*(\G_\rho)$. In \cite{prune} a formula is obtained on the single-crossover case, that is when $\G_\rho$ only contains partitions of the type $\{ \{ 1,..,.m\},\{m+1,...,n\} : 1 < m < n \}$. The main idea is to code the embedded jump chain of the fragmentation process by taking advantage of the fact that once the sites split they become independent. To this end, it is defined the notion of \textit{segmentation trees}, which is a family of graphs serving to this purpose. 

\medskip

In this section, we introduce another family of graphs that will be used in the same manner but for general partitions. These graphs are called \textit{fragmentation trees}, and we construct them in two steps.

\medskip

First we consider a rooted tree $T=(\GG, E, \delta_0)$, with set of nodes $\GG$, set of edges $E$ and with root $\delta_0 \in \G_\rho$. The nodes and the edges fulfill the following properties, called (ORT), for original rooted tree:  
\begin{itemize}
    \item Every $\alpha \in \GG$ fulfills that $\alpha \in \AS(U)$ for some $U \subseteq I$ and $\alpha \not = \{ U \}$.
    \item For all $(\alpha,\beta) \in E \subseteq \GG^2$, the child $\beta$ is a fragmentation of $\alpha$, that is, there $\exists L = L(\alpha, \beta) \in \alpha, \epsilon \in \G_\rho$, $\epsilon|_{L} \not = \{ L \}$ such that $\epsilon|_{L} = \beta$
    \item The atom $L$ is unique between siblings. That is, when $\beta_1,\beta_2$ are two children of $\alpha$ we have $L(\alpha,\beta_1) \not = L(\alpha,\beta_2)$
    \item $\delta_0$ is indeed the root of $T$. That is, $\delta_0 \in \GG$ and for all $(\alpha,\beta) \in E$ we have $\beta \not = \delta_0$.
    \item The graph given by $T$ is indeed a rooted tree. Meaning that, for all $\alpha \in \mathbb{G} \setminus \{ \delta_0 \}$ we have there is a finite $n \in \mathbb{N}$ and $(\alpha_i)_{i=0}^n \subseteq \mathbb{G}$ with $(\alpha_i, \alpha_{i+1}) \in E$ for all $0 \leq i \leq n-1$, such that $(\delta_0, \alpha_0), (\alpha_n, \alpha) \in E$.
\end{itemize}

The properties of (ORT) imply that every node $\alpha \in \GG$ is a restriction of a partition of $\Y^*(\G_\rho)$ to some of the atoms of the root $\delta_0$. It is also deduced that $\alpha$ can have at most $|\alpha|$ children. We stress that the graph $T$ is a tree. We call it the original tree, and the nodes in $\GG$ and edges in $E$ are called the original nodes and the original edges, respectively.

\medskip

We notice there is a partial order of the nodes in any tree fulfilling the properties (ORT). This order is given by the path from the root. That is, for $\alpha, \beta \in \mathbb{G}$ we say that $\alpha \preceq_\GG \beta$, if there is a finite $n \in \mathbb{N}$ and $(\alpha_i)_{i=0}^n \subseteq \mathbb{G}$ with $(\alpha_i, \alpha_{i+1}) \in E$ for all $0 \leq i \leq n-1$, such that $(\alpha, \alpha_0), (\alpha_n, \beta) \in E$. In that case we say that $\beta$ hangs from $\alpha$. We note $\delta_0$ is minimal for $\preceq_\GG$, that is, $\delta_0 \preceq_\GG \alpha$ for all $\alpha \in \GG$.

\medskip

We also include in the definition of (ORT) the degenerate case $\GG = \emptyset$. In that case, $T$ is just the empty graph.

\medskip

Now we will make some modifications to this tree. If $T$ is the empty graph then the modification will give us a tree that just consist on a single node $\{ I \}$. In any other case we will add some extra nodes and extra edges. We refer to the extra edges as branches. First, we add extra nodes and connect them with a new branch to every original node, in such a way that every $\alpha \in \GG$ has exactly $|\alpha|$ children. The new nodes connected to $\alpha$ are given by the elements contained on its ancestor which does not contain sites of any of their siblings. That is, for every $\alpha \in \GG$, the new nodes are given by the set:
\[ \mathcal{N}_\alpha := \{ a \in \alpha | \not \exists \beta \in \GG, \text{ with } (\alpha,\beta) \in E \text{ such that } L(\alpha,\beta) = a \} \]
Finally, we add the extra node $r=\{I\}$ and connect it to $\delta_0$ with a branch. That is, we add the branch $(r, \delta_0)$. Hence the branches are given by:
\[ \mathfrak{B} = \{(r,\delta_0)\} \bigcup \left( \bigcup_{\alpha \in \GG} \bigcup_{a \in \mathcal{N}_\alpha} \{(\alpha,a)\} \right) \]
This construction gives a new tree that responds to the following definition.

\begin{definition} \label{def:tort}
The fragmentation trees are the family of graph obtained by the procedure described above. We start with a tree $T=(\GG, E, \delta_0)$ that fulfills (ORT) and then we modify it using the last algorithm, obtaining a tree noted $T^I=(\hat{\GG},\hat{E},\delta_0)$. We have $\GG \subseteq \hat{\GG}$ and $E \subseteq \hat{E}$, so when working with $T^I$ we refer to $\GG$ and $E$ as the original nodes and edges, respectively.
\end{definition}
We stress that in the degenerate of $T$ being the empty tree then $T^I$ consist of the graph with a single node $\hat{\GG} = \{ I \}$ and no edges or branches. We denote by $\aL_\GG$ the set of leaves of the fragmentation tree $T^I$, so that $\hat{\GG}=\GG \cup \aL_\GG \cup r$. From our algorithm and conditions (ORT) it follows that $\aL_\GG$ defines a partition of $I$.

\medskip

In figure \ref{fig:EJ} we supply an example of a tree and how it is modified to get a fragmentation tree. In this example $I=\{1,2,3,4,5,6,7\}$, $\delta_0 = \{ \{1,6,7\},\{2,3,4\},\{5\} \}$ and the rates are such that $\{ \gamma_1, \gamma_2, \gamma_3 \} \subseteq \mathcal{G}_\rho$ with $\gamma_1 = \{ \{1,6,7\},\{2,3,4\},\{5\}\}$, $\gamma_2 = \{\{1\},\{2,3\},\{4,5\},\{6,7\}\}$, $\: \: \: \: \: \: \: \: \: \: \: \: \: \: \: $ 
\linebreak $\gamma_3 = \{\{1,2,3,4\},\{5\},\{6\},\{7\}\}$. 

\begin{figure}[h]
\begin{subfigure}[b]{0.475\textwidth}
        \centering
        \resizebox{\linewidth}{!}{
            \begin{tikzpicture}[auto, node distance=2.5cm, every loop/.style={},thick, main node/.style={ellipse, draw,font=\sffamily\bfseries}]

  \node[main node] (0) []{$\{ \{1,6,7\}, \{2,3,4\},\{5\}\}$};
  \node[main node] (1) [below of=0, xshift= - 2 cm] {  $\{\{1\},\{6,7\}\}$ };
  \node[main node] (2) [below of=1] {$\{\{ 6\}\{7\}\}$ };
  \node[main node] (3) [below of=0, xshift= 2 cm] {$\{\{2,3\},\{4\}\}$};

  \path[every node/.style={font=\sffamily\small}]
    (0) edge node [left] {} (1)
        edge node [left] {} (3)
    (1) edge node [left] {} (2);
\end{tikzpicture}

        }
        \caption{Tree $T$ fulfilling (ORT)}
        \label{fig:SUB1}
    \end{subfigure}
    \begin{subfigure}[b]{0.475\textwidth}
    \centering
        \resizebox{\linewidth}{!}{
           \begin{tikzpicture}[auto, node distance=2.5cm, every loop/.style={},thick, main node/.style={ellipse, draw,font=\sffamily\bfseries}]

  \node[main node] (11) [yshift=0.75 cm]{$\{ \{1,2,3,4,5,6,7\} \}$};
  \node[main node] (0) [below of=11, yshift=0.75 cm]{$\{\{1,6,7\}, \{2,3,4\},\{5\}\}$};
  \node[main node] (1) [below left of=0, xshift = -3 cm] {$\{\{1\},\{6,7\}\}$ };
  \node[main node] (2) [below of=1,xshift = 1.5 cm,  yshift = 0.5 cm] {$\{\{ 6\},\{7\}\}$ };
  \node[main node] (3) [below of=0, yshift = 0.6 cm] {$\{\{2,3\},\{4\}\}$};

  \node[main node] (4) [below of=0, xshift = 3 cm, yshift= 0.6 cm]{$\{5\} $};
  \node[main node] (5) [below of=1, xshift = -1.5 cm, yshift= 0.5 cm]{$\{1 \}$};
  \node[main node] (6) [below of=2, xshift = -1.5 cm, yshift= 1 cm]{ $\{6 \}$ };
  \node[main node] (7) [below of=2, xshift = 1.5 cm, yshift= 1 cm]{$\{7 \}$};
  \node[main node] (8) [below of=3, xshift = -1 cm, yshift= 1 cm]{$\{ 2,3 \}$};
  \node[main node] (9) [below of=3, xshift = 1 cm, yshift= 1 cm]{$\{ 4 \}$};

  \path[every node/.style={font=\sffamily\small}]
    (0) edge node [left] {}  (1)
        edge node [left] {} (3)
        edge[red] node [right] {} (4) 
    (1) edge node [left] {} (2)
        edge[red] node [left] {} (5) 
    (2) edge[red] node [left] {} (6) 
        edge[red] node [right]{} (7) 
    (3) edge[red] node [left] {} (8) 
        edge[red] node [right] {} (9) 
    (11) edge[red] node [left] {} (0);
\end{tikzpicture}
        }
        \caption{Fragmentation tree $T^I$, branches are in red.}   
        \label{fig:SUB2}
    \end{subfigure}
\caption{In (a) $T$ is a tree that fulfill (ORT), the children of the root are made by using $\gamma_2$ but they could also use other partition. In (b) we show the fragmentation tree $T^I$ associated to $T$, and it is the case that $\mathcal{L}_\mathbb{G} = \{\{1\},\{2,3\},\{4\},\{5\},\{6\},\{7\}\}$.}
\label{fig:EJ}
\end{figure}
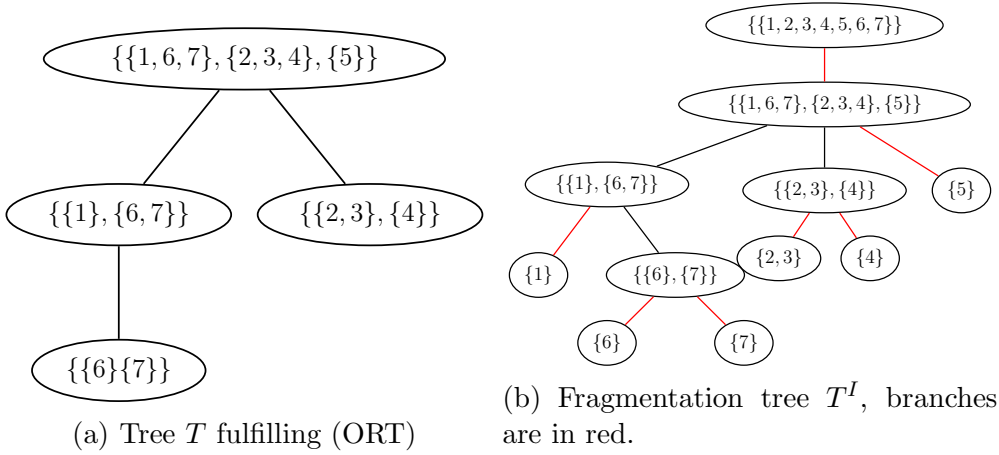

\medskip

\subsection{Formula for the law}
Let us use fragmentation trees to define some distinguished events of the fragmentation process $(X_t)_{t \geq 0}$. We recall that if $\alpha \in \GG$ then there exists $J \subseteq I$ such that $\alpha \in \AS(J)$ and $\alpha \not = \{ J \}$. So, we can define the hitting time:
\[ \mathcal{T}_\alpha = \min \{ t \geq 0 : X_t |_J = \alpha \}.\]
which is the first time the fragmentation process restricted to $J$ hits $\alpha$. With this, for $U \subseteq I$ the following hitting time can be defined:
\[ \mathcal{T}_U = \min\{ \mathcal{T}_\beta : \beta \in \AS(U), \exists \gamma \in \G_\rho, \gamma \not = \{ U \}, \: \text{such that} \: \gamma|_U = \beta\}, \]
which is the first time the sites on $U$ are fragmented. Both $\mathcal{T}_\alpha$ and $\mathcal{T}_U$ are stopping times. Finally we define the events:
\[Max_t(\GG)= \{ \max\{ \mathcal{T}_\alpha: \alpha \in \GG \} \leq t < \min\{\mathcal{T}_{L}: L \in \aL_\GG \} \} ,\]
\[ Or(\GG)= \bigcap_{\alpha \in \GG} \{ \mathcal{T}_\alpha = \min\{ \mathcal{T}_\beta : \beta \in \GG_\alpha \}  \}. \]
Here $\GG_\alpha$ is the set of nodes hanging from $\alpha$ in the original tree $T$, that is:
\[ \GG_\alpha = \{ \beta \in \GG | \alpha \preceq_\GG \beta \}. \]
The event $Max_t(\GG)$ is the one where, up to time $t$, the fragmentation process, restricted to each of its atoms, performs exactly the transitions that appear in the fragmentation tree. The event $Or(\GG)$ is such that when following every path from the root to the leaves, the process performs the transitions in the order given by the tree. We will say the tree $T^I$ codes the embedded jump chain of the fragmentation process when both events happen. In this case $X_t = \mathcal{L}_\mathbb{G}$ and for $(\alpha,\beta) \in E$ we have $\mathcal{T}_{\alpha} \leq \mathcal{T}_{\beta}$. This means that the fragmentation process at time $t$ is the set of leaves of the tree and the paths of the embedded jump chain, when following the atoms starting from the root to the set of leaves, evolves according to the order of the hitting times of the nodes of the tree. Therefore, we are interested in the event:
\[ \F_t (T^I) := Max_t (\GG) \cap Or(\GG). \]
Hence, if $\mathfrak{T}(\alpha)$ is the set of all fragmentation trees satisfying $\aL_\GG = \alpha$, we have
\[ \PP(X_t = \alpha) = \sum_{T^I \in \mathfrak{T}(\alpha)} \PP(\F_t(T^I)), \]
because $\{ X_t = \alpha \}$ is a disjoint union of the class of events $\{ \F_t(T^I)\}_{T^I \in \mathfrak{T}(\alpha)}$. For computing $\PP(\F_t(T^I))$ we require to introduce some additional notation and prove some properties of the fragmentation process.

\begin{definition} \label{margrate}
For $S \subseteq I$ and $\gamma' \in \AS(S)$ we define the marginalized recombination rate as:
\begin{equation} \label{marginalrate}
\rho_{\gamma'}^S := \sum_{ \substack{\gamma \in \G_\rho: \\ \gamma|_S = \gamma'}} \rho_\gamma. \end{equation}
\end{definition}

\begin{proposition}

The fragmentation process is consistent by marginalization. That is: for $S \subseteq I$ the restricted process $(X_t|_S)_{t \geq 0}$ is a fragmentation process with state space $\Y^*(\G_\rho|_S)$ and with rates $(\rho_{\gamma}^S)_{\gamma \in \G_\rho|_S}$, where $\G_\rho|_S$ are the elements of $\G_\rho$ restricted to $S$.

\end{proposition}

\begin{proof}

To this end we use lumping of Markov chains (see \cite{lump}). Let us introduce this concept. For $S \subseteq I$, $\sim_S$ is the relation on $\Y^*(\G_\rho)$ given by:
\[  \forall \delta,\delta' \in \Y^*(\G_\rho): \: \:\delta \sim_S \delta' \Leftrightarrow \delta|_S = \delta'|_S. \]
It is straightforward to check that $\sim_S$ is an equivalence relation. We note by $\Y^*(\G_\rho)/\sim_S$ the set of equivalence classes, which is canonically identified with $\Y^*(\G_\rho|_S)$. We note by $[\delta]$ the equivalence class of $\delta$. Hence, a partition restricted to $S$ on the equivalence class is identified with the restricted partition to $S$ and then the processes $(X_t|_S)_{t \geq 0}$ and $([X_t])_{t \geq 0}$ taking values on $\Y^*(\G_\rho) / \sim_S$ are also identified. In order that they satisfy the Markov property we need to check that $Q_{\delta, [\delta']} := \sum_{\hat{\delta} \in [\delta']} Q_{\delta, \hat{\delta}} $ is equal to $Q_{\delta_1,[\delta']}$ for every element $\delta_1 \in [\delta]$. When this property is satisfied, it is straightforward that $Q_{[\delta],[\delta']} := Q_{\delta,[\delta']}$ is the generator of the process $(X_t|_S)_{t \geq 0}$. 

\medskip

Let us check that property. Take $\delta,\delta' \in \Y^*(\G_\rho)$ such that $\delta|_S \not = \delta'|_S$. If $\delta|_S \not \lt  \delta'|_S$ it is clear that $Q_{\delta,[\delta']} = 0$. So, let $\delta|_S \lt \delta'|_S$ and take $\delta_1 \in [\delta]$. We have
\[ Q_{\delta_1, [\delta']} := \sum_{\hat{\delta} \in [\delta']} Q_{\delta_1, \hat{\delta}} = \sum_{\hat{\delta} \in [\delta']} \sum_{ \substack{\gamma \in \G_\rho: \\ \delta_1 \lt_a^{\gamma} \hat{\delta}}} \rho_{\gamma} =  \sum_{\substack{ \hat{\delta} \in \Y^*(\G_\rho): \\ \hat{\delta} |_S = \delta'|_S  }} \sum_{\substack{\gamma \in \G_\rho: \\ \delta_1 \lt_a^\gamma \hat{\delta}}} \rho_\gamma = \sum_{\substack{\gamma' \in \G_\rho: \\ \delta|_S \lt_{a}^{\gamma'} \delta'|_S }} \sum_{\substack{\gamma \in \G_\rho: \\ \gamma|_S = \gamma'}} \rho_\gamma.\]
In the last equality we have used that, since $\delta|_S \not = \delta'|_S$ the atoms $a(\delta,\delta')$ and $a(\delta_1,\hat{\delta})$ only contain sites of $S$. Since the last expression does not depend on $\delta_1$ but only on $\delta$ we get that $Q_{\delta_1, [\delta']} = Q_{\delta, [\delta']}$. Moreover:
\[ Q_{\delta_1,[\delta']} =  \sum_{\substack{\gamma' \in \G_\rho: \\ \delta|_S \lt_{a}^{\gamma'} \delta'|_S }} \sum_{\substack{\gamma \in \G_\rho: \\ \gamma|_S = \gamma'}} \rho_\gamma = \sum_{\substack{\gamma' \in \G_\rho: \\ \delta|_S \lt_{a}^{\gamma'} \delta'|_S }} \rho_{\gamma'}^S, \]
and so the process $(X_t|_S)_{t \geq 0}$ have exactly the generator of a fragmentation process with the marginalized rates.
\end{proof}

A a direct consequence of the last proposition is the following result.
\begin{proposition} \label{prop:FMA}
Let $S \subseteq I$. Then for all $\delta \in \Y^*(\G_\rho)$ such $\{ S \} \in \delta$ we have:
\begin{equation*}\label{FROMMARG}
\PP_\delta (X_t |_S = \{ S \})= \exp \left(-t \sum_{\substack{\gamma \in \G_\rho: \\ \gamma|_S \not = \{ S \}}} \rho_\gamma \right),
\end{equation*}
quantity that does not depend on the partition $\delta$ with $\{S\} \in \delta$.
\end{proposition}
We note that Proposition \ref{prop:FMA} is just a rigorous statement of Remark \ref{rem:ia}. Proposition \ref{prop:FMA} allows to define the function $\lambda_S^S (t) := \PP_{\delta}(X_t|_S = \{ S\})$ which is independent of $\delta$ such that $\{ S \} \in \delta$. This is the exponential property of the holding time of the marginal fragmentation process in $\{ S \}$. Now, since the fragmentation process acts independently on each of its atoms we can extend the definition of $\lambda$ for every $\delta \in \AS(I)$ and $S \subseteq I$ as:
\[ \lambda_{\delta}^{S}(t) := \PP(X_t|_S = \delta \, | \, X_0 = \delta \cup \{S^c \})=\prod_{a \in \delta} \lambda_a^a(t).\]
We will use $\lambda_\delta^S := \lambda_\delta^S (1)$. From Proposition \ref{prop:FMA} we have that $\lambda_\delta^S(t) = (\lambda_\delta^S)^t$. At this point, we require the notion of erasing some part of a tree.

\begin{definition}
Given a tree $T=(\GG,E,\delta_0)$ fulfilling (ORT), $\alpha \in \GG, H \subseteq E$, we denote by $T^I_\alpha(H)$ the fragmentation tree when erasing $H$ from the subtree with root $\alpha$ of $T$, and using the algorithm to transform it to a fragmentation tree in the same manner as Definition \ref{def:tort}. We will denote by $I_\alpha := \bigcup_{L \in \alpha} L$ the sites that are present on the tree, and by $\aL_{\GG_\alpha(H)}$ the leaves of this tree. It follows that $\aL_{\GG_\alpha(H)}$ is a partition of $I_\alpha$.
\end{definition}
In Figure \ref{fig:EJA} we provide an example of these trees, based on the tree in Figure \ref{fig:EJ}.

\medskip

\begin{figure}[h]
\begin{subfigure}[b]{0.4\textwidth}
       \centering
        \resizebox{\linewidth}{!}{
           \begin{tikzpicture}[auto, node distance=2.5cm, every loop/.style={},thick, main node/.style={ellipse, draw,font=\sffamily\bfseries}]

  \node[main node] (0) []{$\{ \{1,6,7\} \}$};
  \node[main node] (1) [below left of=0, xshift = 1.75 cm] {$\{ \{1\},\{6,7\} \}$ };
  \node[main node] (6) [below of=1, xshift = -1.5 cm, yshift= 1 cm]{ $\{1 \}$ };
  \node[main node] (7) [below of=1, xshift = 1.5 cm, yshift= 1 cm]{$\{6,7 \}$};
  
  \path[every node/.style={font=\sffamily\small}]
    (0) edge node [] {}  (1)
    (1)edge[red] node [left] {} (6) 
       edge[red] node [right] {} (7); 
\end{tikzpicture}
}
        \caption{Example 1: $T_{\alpha_1}^{I}(H_1)$.}
        \label{fig:SUB1A}
    \end{subfigure}
    \begin{subfigure}[b]{0.6\textwidth}
    \centering
        \resizebox{\linewidth}{!}{
           \begin{tikzpicture}[auto, node distance=2.5cm, every loop/.style={},thick, main node/.style={ellipse, draw,font=\sffamily\bfseries}]

  \node[main node] (11) [yshift=0.75 cm]{$\{ \{1,2,3,4,5,6,7\} \}$};
  \node[main node] (0) [below of=11, yshift=0.75 cm]{$\{ \{1,6,7\}, \{2,3,4\}, \{5\} \}$};
  \node[main node] (1) [below left of=0, xshift = -2 cm] {$\{1,6,7\}$ };
  \node[main node] (3) [below of=0, yshift = 0.6 cm] {$\{2,3,4\}$};

  \node[main node] (4) [below of=0, xshift = 3 cm, yshift= 0.6 cm]{$\{5\} $};

  \path[every node/.style={font=\sffamily\small}]
    (0) edge[red] node [left] {}  (1)
        edge[red] node [left] {} (3)
        edge[red] node [right] {} (4) 
    (11) edge[red] node [left] {} (0);
\end{tikzpicture}
        }
        \caption{Example 2: $T_{\alpha_2}^{I}(H_2)$.}   
        \label{fig:SUB2A}
    \end{subfigure}
\caption{In this example we use the same tree as in Figure \ref{fig:EJ} for showing $T_\alpha^I(H)$. In the first example $\alpha_1 = \{ \{1\},\{6,7\}\}$ and $H_1 = \{(\{\{1\},\{6,7\}\},\{\{6\},\{7\}\})\}$. In the second one $\alpha_2$ is the original root and $H_2$ are all the edges that are incident to it.}
\label{fig:EJA}
\end{figure}
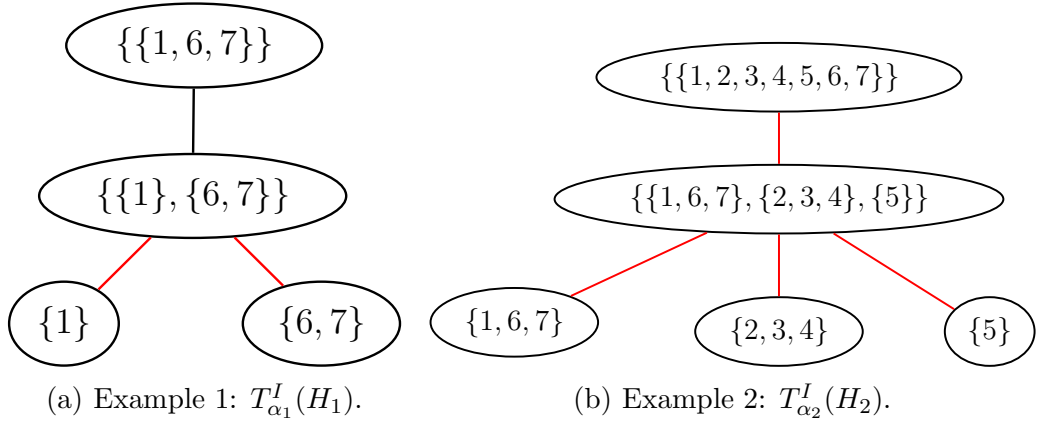

Now for a fragmentation tree $T^I = (\hat{\GG},\hat{E}, \delta_0)$, $\alpha \in \GG$ and $H \subseteq E$ we define:
\[ \lambda_{\GG_\alpha (H)}^{I_\alpha} := \lambda_{\aL_{\GG_\alpha}(H)}^{I_\alpha}. \]
By using the Proposition \ref{prop:FMA} it is easy to check that if we have $J,K \subseteq I$ with $J \cap K = \emptyset$ and $\delta \in \AS(J), \delta' \in \AS(K)$ then:
\[ \lambda_\delta^{J} (t) \lambda_{\delta'}^{K}(t) = \lambda_{\delta \cup \delta'}^{J \cup K} (t). \]
\begin{remark}
As seen in \cite{lyu} the functions $(\lambda_\delta^I)_{\delta \in \Y^*(\G_\rho)}$ are the eigenvalues of the semigroup of the fragmentation process. 
\end{remark}
We have all the elements that are necessary to supply a formula for the law of the fragmentation process.
\begin{theorem} \label{LAWF}
Let $T^I =(\hat{\GG},\hat{E},\delta_0)$ be a fragmentation tree. If $|\GG| = 0$, then $\PP(\F_t(T^I)) = (\lambda_I^I)^t$. Assume $|\GG|>0$. Let us suppose $\lambda_{\GG_\alpha(H)}^{I_\alpha} \not = \lambda_{I_\alpha}^{I_\alpha}$ for all $\alpha \in \GG, H \subseteq E$. Then:
\begin{equation*}\PP(\F_t(T^I)) = \sum_{H \subseteq E} (-1)^{|H|} [ (\lambda_{\GG_{{\delta_0}(H)}}^{I} )^t - (\lambda_{I}^{I} )^t] \prod_{\beta \in \GG} \frac{ \rho^{I_\beta}_\beta  }{ \log( \lambda_{\GG_\beta(H)}^{I_\beta } ) - \log( \lambda_{ I_\beta  }^{I_\beta } ) },
\end{equation*}
where $\rho_\beta^{I_\beta}$ is the marginalized recombination rate defined in (\ref{marginalrate}).
\end{theorem}

\begin{proof}
The proof goes by induction. The case $|\GG|=0$ follows from definition. We proceed to the case $|\GG|=1$. By the Markov property, the definition of $\F_t(T^{I})$ and the fact that on the fragmentation process the sites which have been split are independent, we obtain
\[ \mathbb{P}(\mathcal{F}_t(T^{I})) = \int_{0}^{t} \mathbb{P}(X_{T_1} = \delta_0) \left[\prod_{a \in \delta_0} \mathbb{P}_{\delta_0} (\mathcal{F}_{t-u}(T^{I_{a}})) \right] d \mathbb{P}(T_1 = u). \]
Note that by the definition of the fragmentation process we have:
\[ \mathbb{P}(X_{T_1} = \delta_0) = \frac{ \rho_{\delta_0}  }{ |\rho| }  \: \: \: \: \text{and} \]
\[ d \mathbb{P}(T_1 = u) = |\rho| e^{- |\rho| u} d u = |\rho| \lambda_I^I(u) du. \]
So, we get
\begin{equation} \label{INDUCT}  \mathbb{P}(\mathcal{F}_t(T^{I})) = \int_{0}^{t}  \rho_{\delta_0}  \left[\prod_{a \in \delta_0} \mathbb{P}_{\delta_0} (\mathcal{F}_{t-u}(T^{I_{a}})) \right]  \lambda_I^I (u) d u. \end{equation}
Moreover, since $|\GG| = 1$ we obtain:
\[  \mathbb{P}_{\delta_0} (\mathcal{F}_{t-u}(T^{I_{a}})) = \mathbb{P}_{\delta_0} (X_{t-u}|_{I_{a}} = \{ I_{a} \} ) = \lambda_{ I_{a} }^{I_{a}} (t-u). \]
Then, by replacing these terms in formula (\ref{INDUCT}) and by using the properties of the functions $\lambda$ we have that:
\begin{align*} \mathbb{P}(\mathcal{F}_t(T^I)) &= \rho_{\delta_0} \int_0^t \lambda_{I}^I  (u) [\prod_{a \in \delta_0} \lambda_{ I_{a}  }^{I_{a}} (t-u)] du = \rho_{\delta_0} \int_{0}^{t} \lambda_I^I (u) \lambda_{ \cup_{a \in \delta_0} \{ I_{a} \} }^{I} (t-u) du  \\ &= \rho_{\delta_0}  \int_{0}^{t} ( \lambda_I^I)^u ( \lambda_{\delta_0}^{I} )^{t-u} du =  \rho_{\delta_0} \left( \frac{(\lambda_{\delta_0}^{I})^t -  (\lambda_{I}^{I})^t}{ \log(\lambda^I_{\delta_0} )- \log(\lambda_{I}^{I})} \right). \end{align*}
In the calculation of the integral we have used the hypothesis made on the functions $\lambda$. Note that $|\GG|=1$ implies $E = \emptyset$, and so $\rho_{\delta_0} = \rho_{\delta_0}^I$. Hence the last computation coincides with the formula stated in the Theorem. 

\medskip

Now we proceed to the inductive step. First, without lose of generality we can assume that the children of the root $\delta_0$ are not leaves. This does not change the proof and makes the notation easier. By repeating the last computations we arrive to:
\begin{equation} \label{eq:inductivef} \PP(\mathcal{F}_t(T^{I})) = \int_{0}^{t}  \rho^I_{\delta_0}  \left[\prod_{a \in \delta_0} \PP_{\delta_0} (\mathcal{F}_{t-u}(T^{I_{a}})) \right]  \lambda_I^I (u) d u. \end{equation}
By using the induction hypothesis, this quantity is equal to:
\begin{align*} \int_{0}^{t} \rho^I_{\delta_0} (\lambda_I^I)^u \left( \prod_{a \in \delta_0}\sum_{H_a \subseteq E_{a}} (-1)^{|H_a|} [(\lambda_{\GG_{a}(H_i)}^{I_{a}})^{t-u} - (\lambda_{I_{a}}^{I_{a}})^{t-u}] \prod_{\beta  \in \GG_{a}} \frac{\rho^{I_\beta }_\beta }{\log(\lambda_{\GG_\beta (H_a)}^{I_\beta }) - \log(\lambda_{I_\beta }^{I_\beta }) } \right) du, \end{align*}
where $\GG_{a}$, $E_{a}$ are set of nodes and edges of the tree $T^{I_{a}}$, respectively. Now, by using distribution of the sum we arrive to the following expression for $\mathbb{P}(\F_t(T^I))$:
 \begin{align*} \int_{0}^{t} \rho^I_{\delta_0} (\lambda_I^I)^u & \left( \sum_{a \in \delta_0, H_a \subseteq E_{a}} (-1)^{\sum_{a \in \delta_0} |H_a|} \prod_{a \in \delta_0} [ (\lambda_{\GG_{a}(H_a)}^{I_{a}} )^{t-u} - (\lambda_{I_{a}}^{I_{a}})^{t-u} ] \right. \\ & \left. \times \prod_{a \in \delta_0} \prod_{\beta  \in \GG_{a} }  \frac{ \rho^{I_\beta }_\beta  }{ \log(\lambda_{\GG_\beta  (H_a)}^{I_\beta }) - \log(\lambda_{I_\beta }^{I_\beta }) } \right) du. \end{align*}
 Now, use the identity
 \[ \prod_{a \in \delta_0} [ (\lambda_{\GG_{a}(H_a)}^{I_{a}} )^{t-u} - (\lambda_{I_{a}}^{I_{a}})^{t-u} ] =  \sum_{\delta_1 \subseteq \delta_0} (-1)^{|\delta_0|-|\delta_1|} (\lambda_{[\cup_{a \in \delta_1} \GG_{a} (H_a)] \cup [ \cup_{a \in \delta_0 \setminus \delta_1} \{ I_{a} \} ] }^I )^{t-u}, \]
to get the equality
 \begin{align*} &\int_{0}^{t} (\lambda_I^I)^u \prod_{a \in \delta_0} [ (\lambda_{\GG_{a}(H_a)}^{I_{a}} )^{t-u} - (\lambda_{I_{a}}^{I_{a}})^{t-u} ] du \\ &= \sum_{\delta_1 \subseteq \delta_0 } (-1)^{|\delta_0|-|\delta_1|} \int_{0}^{t} (\lambda_I^I)^u (\lambda_{[\cup_{a \in \delta_1} \GG_{a_i} ( H_i)] \cup [ \cup_{a \in \delta_0 \setminus \delta_1} \{ I_{a} \} ] }^I )^{t-u} du \\ &= \sum_{\delta_1 \subseteq \delta_0 } (-1)^{|\delta_0|-|\delta_1|} \frac{(\lambda_{[\cup_{a \in \delta_1} \GG_{a} ( H_a)] \cup [ \cup_{a \in \delta_0 \setminus \delta_1} \{ I_{a} \} ] }^I )^t - (\lambda_I^I)^t}{ \log(\lambda_{[\cup_{a \in \delta_1} \GG_{a} ( H_a)] \cup [ \cup_{a \in \delta_0 \setminus \delta_1} \{ I_{a} \} ] }^I ) - \log(\lambda_I^I) }, \end{align*}
 where we used the hypothesis over $\lambda$ to compute the integral. So, from this expression and by making some manipulation over the sums we get that $\PP(\F_t(T^I))$ is equal to:
 \begin{align}\label{terminofeo} \nonumber \rho^I_{\delta_0}\sum_{\delta_1 \subseteq \delta_0 } \sum_{a \in \delta_0: H_a \subseteq E_{a}} (-1)^{\sum_{a \in \delta_0|} |H_a|}  & (-1)^{|\delta_0|-|\delta_1|} \frac{(\lambda_{[\cup_{a \in \delta_1} \GG_{a} (H_a)] \cup [ \cup_{a  \in \delta_0 \setminus \delta_1} \{ I_{a} \} ] }^I )^t - (\lambda_I^I)^t}{ \log(\lambda_{[\cup_{a \in \delta_1} \GG_{a} (H_a)] \cup [ \cup_{a \in \delta_0 \setminus \delta_1} \{ I_{a_i} \} ] }^I ) - \log(\lambda_I^I) } \\ & \times \prod_{a \in \delta_0} \prod_{\beta  \in \GG_{a} }  \frac{ \rho^{I_\beta }_\beta  }{ \log(\lambda_{\GG_\beta  (H_a)}^{I_\beta }) - \log(\lambda_{I_\beta }^{I_\beta }) }]. \end{align}
Now, we perform the change of variable $H = [\cup_{a \in \delta_0} H_a] \cup [\cup_{a \in  \delta_0 \setminus \delta_1} e_a]$ where $e_a$ is the edge that connects $\delta_0$ with its children with sites $I_{a}$. Then, we get $(\lambda_{[\cup_{a \in \delta_1} \GG_{a} ( H_a)] \cup [ \cup_{i \in \delta_0 \setminus \delta_1} \{ I_{a} \} ] }^I) = \lambda_{\GG_{\delta_0} (H)}^{I}$. Observe that the sum over $\{e_a: a \in \delta_0 \setminus \delta_1 \}$ and over the set of edges $H$ runs over all $E$. Finally, we use 
\[ |H| = |\cup_{a \in \delta_0} H_a| + | \cup_{a \in \delta_0 \setminus \delta_1} e_a| = \sum_{a \in \delta_0} |H_a| + |\delta_0| - |\delta_1|,\]
in the expression (\ref{terminofeo}) to get the result.
\end{proof}

\begin{remark}
Since the solutions of equation (\ref{EDOREC}) depend explicitly on $\PP(\F_t(T^I))$ the last theorem supplies an expression for the solutions. In contrast to the solutions found in \cite{bbs} these formulae are non-recursive, but depend on the structure of the fragmentation trees.
\end{remark}

\begin{remark}
The last theorem can be used to deduce the form of the solution even when the hypothesis on the functions $\lambda$ does not apply. This works as follows. Consider the set of nodes closest to the root that fulfils the hypothesis, that is, 
\[ \mathcal{U} = \{ \beta \in \GG | \exists \alpha \in \GG, \text{ with } (\alpha,\beta) \in E, \text{ and for some } H \subseteq E,  \lambda_{\GG_\alpha(H)}^{I_\alpha} = \lambda_{I_\alpha}^{I_\alpha} \}.  \]
Theorem \ref{LAWF} can be used on every $\beta \in \mathcal{U}$, giving $\PP(\mathcal{F}_{t}(T^{I_\beta}))$. For every node that has its offspring in $\mathcal{U}$ we can extend this formula, using equality (\ref{eq:inductivef}). This can continue by induction until we get to the root, giving $\PP(\mathcal{F}_{t}(T^{I}))$. We do not state this general formula because we were unable to find a pattern that works in a general way.
\end{remark}

\begin{remark}
Some conditions can be stated under which the formula for Theorem \ref{LAWF} applies. The simpler one, and that follows from the definition of the functions $\lambda$ and Proposition \ref{prop:FMA}, is for the quantities $(\rho_{\delta})_{\delta \in \mathcal{G}_\rho}$ to be linearly independent over $\mathbb{Z}$. That is, for every with $\delta_1,\delta_2,..., \delta_k \in \mathcal{G}_\rho$ all different, there does not exists $x_1,.., x_k \in \mathbb{Z}$ such that $\sum_{i=1}^{k} \rho_{\delta_i} x_i = 0$.
\end{remark}

\begin{remark}
It can be observed that the structure of the formulae of $\PP(\F_t(T^I))$ is reminiscent of an inclusion-exclusion. This is to be expected, as it has been studied and proven \cite{prune} to be the case for single-crossover. It remains an open problem to give the same interpretation in the general partition framework we have developed.
\end{remark}

\medskip

\section{Limit behaviour} \label{Sec:lim}

First, we start by stating the stationary behaviour of the fragmentation process. We also supply the consequences it has for solutions of equation (\ref{EDOREC}).

\medskip

For this we recall some notation for Markov processes. For $U \subseteq \Y^*(\G_\rho)$ $\tau_U$ denote the time at which $(X_t)_{t \geq 0}$ hits $U$ and for $\delta \in \Y^*(\G_\rho)$ we denote $\tau_\delta := \tau_{\{ \delta \}}$. We are interested in studying $\tau_{\gamma^{\G_\rho}}$, as $\gamma^{\G_\rho}$ is the unique absorbing state for the process. For simplicity we put $\tau:= \tau_{\gamma^{\G_\rho}}$.

\begin{theorem} \label{stationary}
Let $\mu \in \mathcal{P}_I$ and $\bar{\mu} = \bigotimes_{J \in \gamma^{\G_\rho}} \mu_J$. Then $\bar{\mu}$ is a stationary point for $\Xi$, that is, for all $t\geq0$, $ \Xi_t (\bar{\mu}) = \bar{\mu}$. Moreover $\mathbb{P}(\tau < \infty)=1$ and:
\begin{equation}\label{ATRACREC}  \lim_{t \rightarrow \infty} \Xi_t \mu = \bar{\mu}. \end{equation}
\end{theorem}

\begin{proof}
Let us see that $\bar{\mu}$ is a stationary point. Indeed, for every $t \geq 0$, we use Theorem \ref{RECSOL} to get: 
\begin{align*}\label{FIXPOINT} \Xi_t (\bar{\mu}) &= \sum_{\delta \in \mathcal{Y}^*(\mathcal{G}_\rho)} \mathbb{P}(X_t = \delta) \bigotimes_{L \in \delta} \bigotimes_{J \in \gamma^{\G_\rho}} \mu_{J \cap L} \\ &= \sum_{\delta \in \mathcal{Y}^*(\mathcal{G}_\rho)} \mathbb{P}(X_t = \delta) \bigotimes_{J \in \gamma^{\G_\rho}} \mu_J = \bar{\mu} \sum_{\delta \in \mathcal{Y}^*(\mathcal{G}_\rho)} \mathbb{P}(X_t = \delta) = \bar{\mu}, \end{align*}
where we have used $\delta \preceq \gamma^{\G_\rho}$ for all $\delta \in \Y^*(\G_\rho)$, and so:
\[ \bigotimes_{L \in \delta} \mu_{J \cap L} = \begin{cases} \mu_J & J \subseteq L, \\ \mu_{\emptyset} & in \: other \: case. \end{cases} \]
Hence the stationary of $\bar{\mu}$ is proven. Now, from Theorem \ref{RECSOL} we have
\begin{align*}
\Xi_t \mu = \mathbb{E}(\mu^{(t)}) = \mathbb{E}(\mu^{(t)}, \tau \leq t) + \mathbb{E}(\mu^{(t)}, \tau > t) = \bar{\mu} \mathbb{P}(\tau \leq t) + \mathbb{E}(\mu^{(t)}, \tau > t).
\end{align*}
So, to show (\ref{ATRACREC}) it suffices to prove that $\mathbb{P}(\tau < \infty)=1$. But this holds because on one hand once the Markov chain leaves a state it does never return to it, and on the other hand for every $\delta \not = \gamma^{\G_\rho}$ the sojourn time is almost surely finite.
\end{proof}

Now we will describe the quasi-limiting behaviour of the process. That is the study of the process conditioned to not hit its absorbing state. Quasi-limiting behaviour, along with a similar aspect that is quasi-stationarity, has been extensively studied for aperiodic and irreducible Markov chains. There are plenty of conditions that are sufficient to have existence and formulas of quasi-limiting distributions when avoiding some class of states, see \cite{sjc}. The information supplied by quasi-limiting distributions has a meaning depending on the set of forbidden states. In the study of population dynamics, quasi-limiting distributions appear naturally when a population is conditioned to avoid extinction. This is the context for the processes studied in \cite{quasiex}, and this happens in the vast majority of the literature devoted to populations dynamics. Also, the processes take values on $\NN$ and $\RR$, because they count the number of individuals, or take values on point measure sets as in \cite{cmmsm}. 

\medskip

There is some literature on quasi-limiting behaviour on reducible Markov chains, although much more sparse. We can mention \cite{pr2} for some interesting results and \cite{pr1} for a survey on quasi-limiting behaviour. In this literature conditions under which quasi-limiting distribution exists, and expression for these distributions, are given. However, the hypothesis needed are too strong on our case. In our context, they would translate into asking the Markov chain to have a unique state that achieves the maximal sojourn rate. Since we do not want to impose this restriction we are obliged to look for new techniques to understand the quasi-limiting behaviour of the fragmentation process.

\medskip

The difference in our work with these studies is that our process takes values on $\Y^*(\G_\rho)$, a set having a very special hierarchical structure. Furthermore, the process does never returns to a state that it leaves, which is different from the irreducibility hypothesis used in \cite{sjc} or in \cite{pp}. However, this same hierarchical structure is what gives a chance to generalizing the results from \cite{pr1}, \cite{pr2}.

\medskip

The study of the long time behaviour of the fragmentation process is interesting when we contrast it with our other results. For instance, Theorem \ref{LAWF} gives an insight of the fragmentation process law but it does not give a clue on its asymptotic behaviour since it seems unfeasible to take limit on the formula. On the other hand, with respect to Remark \ref{fragmeaning} our study answers to the question: Which is the shape of the genetic material of individuals, backwards in time,  when we condition to the fact that some genes can still be separated in recombination?

\medskip

Our study of the quasi-limiting behaviour uses a similar schema as the one developed in \cite{sm}. 

\medskip

In what follows we assume that $|\G_\rho|>1$ so the process is non-trivial. Following Theorem \ref{stationary} it is expected that after a long time the process arrives at the absorbing state $\gamma^{\G_\rho}$, this is the stationary behaviour. When this has not happened at some big-time $t$, it is expected that the process is in some state connected to the absorbing state and having the highest sojourn rate. This is just because having the highest sojourn is the expected way for the process to avoid hitting $\gamma^{\G_\rho}$ at all. With this in mind we define the set of states that can arrive to the absorbing state:
\[ \Delta = \{ \delta \in \Y^*(\G_\rho): \delta \lt \gamma^{\G_\rho}, \delta \not = \gamma^{\G_\rho} \}. \]
The highest sojourn rate on this set is given by:
 \[ \eta = - \max \{ Q_{\delta, \delta}: \delta \in \Delta\}. \]
We denote by $\V$ the set of states that have maximal sojourn rate on $\Delta$, that is:
\[ \V  = \{ \delta \in \Delta: Q_{\delta,\delta} = - \eta \}. \]
If $|\V| =1$ then we could use Theorem 5 of \cite{pr1} to easily get the quasi-limiting behaviour. However we do not want this restriction, so we only assume the obvious condition $|\V| \geq 1$.  We require to consider the highest sojourn rate outside $\V$:
\[  \beta_0 = -\max\{ Q_{\delta,\delta}: \delta \in \Y^*(\G_\rho), \delta \not = \gamma^{\G_\rho}, \delta \not \in \V \}. \]
As usual we we denote $P_{\delta,\delta'}^{t} = \PP( X_t = \delta' \, | \, X_0 = \delta)$ for $\delta, \delta' \in \Y^*(\G_\rho), t \geq 0$. We have the following result on sojourn probabilities.
\begin{lemma}\label{FSUM} 
Let $\delta \in \Y^*(\G_\rho)$, $\delta \not = \gamma^{\G_\rho}$, $\delta \not \in \Delta$. Then, for all $\delta' \in \Delta$ such that $ \PP_{\delta}(\tau_{\delta'} < \infty) > 0, $ we have:
\begin{enumerate}[label={(\roman*)}]
    \item $Q_{\delta,\delta} < Q_{\delta',\delta'}$
    \item $\forall \tilde{\delta} \in \V: \: -Q_{\tilde{\delta},\tilde{\delta}} =Q_{\tilde{\delta},\gamma^{\G_\rho}}$ and $t \geq 0 , \:  P_{\tilde{\delta},\tilde{\delta}}^t + P^t_{\tilde{\delta}, \gamma^{\G_\rho}} = 1.$
\end{enumerate}
\end{lemma}
\begin{proof}
$(i)$: Since $\delta' \lt \gamma^{\mathcal{G}_\rho}$ then there is a unique $ \bar{a} \in \delta'$ such that $\delta'|_{\bar{a}} \not = \lt \gamma^{\mathcal{G}_\rho}|_{\bar{a}}$. Also, $\delta \preceq \delta'$ implies that
\[ \delta \lt_a^\gamma \delta \Rightarrow \delta' \lt_{a'}^{\gamma} \: \: \: \forall a' \subseteq a,\]
or equivalently
\[ \delta' \not \lt_{a'}^{\gamma} \delta' \Rightarrow \delta \not \lt_a^{\gamma} \delta \: \: \: \forall a' \subseteq a \]
Moreover, $\delta \preceq \delta'$ with $\delta \not = \delta'$ implies the existence of $\gamma_1 \in \G_\rho$ such that $\delta \not \lt_{\gamma_1}^{a_1} \delta$ for all $a \subseteq a_1$. In particular $a_1 \cap \bar{a} = \emptyset$. With these relations we get that
\begin{align*} - Q_{\delta' ,\delta'} &= \sum_{\gamma \in \mathcal{G}_\rho} \sum_{\substack{ a \in \delta': \\ \delta' \not \lt_a^{\gamma} \delta'}} \rho_\gamma = \sum_{\substack{\gamma \in \mathcal{G}_\rho: \\ \delta' \not \lt_{\bar{a}}^{\gamma} \delta'}} \rho_\gamma  \\ &< \sum_{\substack{\gamma \in \mathcal{G}_\rho: \\ \delta' \not \lt_{\bar{a}}^{\gamma} \delta' }} \rho_\gamma  + \rho_{\gamma_1} \leq \sum_{\gamma \in \mathcal{G}_\rho} \sum_{\substack{ a \in \delta: \\ \delta \not \lt_a^\gamma \delta }} \rho_\gamma = -Q_{\delta,\delta}. \end{align*}
Hence $Q_{\delta, \delta} < Q_{\delta', \delta'}$.

\medskip

$(ii)$: Let us proceed by contradiction. Let $\delta \in \mathcal{V}$ and consider there is $\delta' \not = \delta$ such that $\delta \lt \delta'$, $\delta' \not = \gamma^{\G_\rho}$ and $Q_{\delta, \delta'} > 0$.  Given $\gamma^{\G_\rho}$ is an absorbent state, there must exist $\hat{\delta} \in \Delta$ such that $\mathbb{P}_{\delta'}(\tau_{\hat{\delta}} < \infty) > 0$. Note that, given that once the process exits a state it never return to it is clear that  $\delta \not = \hat{\delta}$. 

\medskip
Given this elements there is $\gamma' \in \G_\rho$ and $a_\delta \in \delta$ such that $\delta \lt_{a_\delta}^{\gamma'} \delta'$, with $\delta|_{a_\delta} \not = \gamma^{\G_\rho}|_{a_\delta}$ and $\hat{\delta} \lt_{\hat{a}}^{\gamma'} \hat{\delta}$ for all $\hat{a} \in \hat{\delta}$, $\hat{a} \subseteq a_\delta$. Also as $\delta \in \Delta$ we have that $\delta|_a = \gamma^{\G_\rho}|_a$ for all $a \in \delta$ with $a \not = a_\delta$. Same thing applies to $\hat{\delta}$, there is a unique $a_{\hat{\delta}}$ such $\hat{\delta}|_{a_{\hat{\delta}}}  \not = \gamma^{\G_\rho}|_{a_{\hat{\delta}}}$. It follows that $a_{\hat{\delta}} \subseteq a_\delta$. Hence:
\begin{align*} \eta = -Q_{\delta, \delta} &= \sum_{\gamma \in \mathcal{G}_\rho} \sum_{\substack{a  \in \delta: \\ \delta \not \lt_a^\gamma \delta}} \rho_\gamma = \sum_{\substack{\gamma \in \mathcal{G}_\rho: \\ \delta \not \lt_{a_\delta}^{\gamma} \delta}}  \rho_{\gamma}  = \rho_{\gamma'} + \sum_{\substack{\gamma \in \mathcal{G}_\rho, \gamma \not = \gamma': \\ \delta \not \lt_{a_\delta}^\gamma \delta}}  \rho_\gamma  \\  &> \sum_{\substack{\gamma \in \mathcal{G}_\rho, \gamma \not = \gamma': \\ \delta \not \lt_{a_\delta}^{\gamma} \delta}}  \rho_\gamma \geq \sum_{\substack{\gamma \in \mathcal{G}_\rho: \\ \hat{\delta} \not \lt_{a_{\hat{\delta}}}^{\gamma} \hat{\delta} }}  \rho_\gamma =  -Q_{\hat{\delta}, \hat{\delta}}, \end{align*}
which implies $- \eta = \max \{ Q_{\delta, \delta}: \delta \in \Delta\} < Q_{\hat{\delta},\hat{\delta}}$ with $\hat{\delta} \in \Delta$ which is a clear contradiction. Thus, we get that for $\tilde{\delta} \in \V, -Q_{\tilde{\delta},\tilde{\delta}} = Q_{\tilde{\delta}, \gamma^{\G_\rho}}$. Hence when starting from $\tilde{\delta}$ we have $P_{\tilde{\delta},\tilde{\delta}}^t + P_{\tilde{\delta}, \gamma^{\G_\rho}}^t = 1$ for all $t > 0$.
\end{proof}
We need a result that control the decay of the hitting time of the set $\V \cup \{\gamma^{\G_\rho} \}$. This is analogues to Lemma 4.2 in  \cite{sm}, but since we are in continuous time now we require a complete and detailed way.
\begin{lemma} \label{decay}
For all $\theta > 0$ there exists a constant $C = C(\theta)$ such that: \[ \PP(\forall u \leq t; X_u \not \in \V \cup \{ \gamma^{\G_\rho} \} ) \leq C (e^{-\beta_0} + \theta)^t. \]
\end{lemma}
\begin{proof} Let $ U = \Y(\mathcal{G}) \setminus (\mathcal{V} \cup \{ \gamma^{\G_\rho} \})$. and suppose $\Y(\mathcal{G}) \setminus (\mathcal{V} \cup \{ \gamma^{\G_\rho} \} \cup \{I\}) \not = \emptyset$ for the result to be non trivial. Fix $\delta_1 = \{ I \}$ and for every $|I| \geq s \geq 2$ consider
\[ \mathcal{C}(U,s) = \{ (\delta_1,..., \delta_s) \in U^s : \forall r \leq s-1, \delta_r \lt \delta_{r+1} \:, \: \delta_r \not = \delta_{r+1} \}. \]
Now, the event \{ $(X_u)_{t \geq u \geq 0} \mapsto (\delta_1, \delta_2 ,.., \delta_s)\}$ is defined by the existence a sequence of times $0 < t_1 < ... < t_{s-2} < t_{s-1} = t$ such that $X_0 = \delta_1, X_{t_1} = \delta_2,..., X_{t_{s-1}} = \delta_s$ and $\{X_u: u \leq t\} = \{ \delta_1,...,\delta_s \}$. Recall that we use $Q$ for the generator of the fragmentation process. Let us define $K = \max\{ -Q_{\delta,\delta} : \delta \in \Y^*(\G_\rho) \}$. Then standard techniques of continuous time Markov processes yields:
\begin{align*} &\PP(\forall u \leq t; X_u \not \in \mathcal{V} \cup \{ D^\rho \}) \\ & \leq \sum_{|I| \geq s \geq 2} \sum_{(\delta_1 ,.., \delta_s) \in \mathcal{C}(U,s)} \int_{0 \leq \sum_{i=1}^{s-1} t_i \leq t, 0 \leq t_i}  e^{-\beta_0(t - \sum_{i=1}^{s-1} t_i) } K^{s-1} \left(\prod_{i=1}^{s-1} e^{- t_i \beta_0}\right) dt_1 ,..., dt_{s-1} \\  &= \sum_{|I| \geq s \geq 2} \sum_{(\delta_1 ,.., \delta_s) \in \mathcal{C}(U,s)} \int_{0 \leq \sum_{i=1}^{s-1} t_i \leq t, 0 \leq t_i}  e^{-\beta_0(t - \sum_{i=1}^{s-1} t_i) } K^{s-1}  e^{- \sum_{i=1}^{s-1} t_i \beta_0} dt_1 ,..., dt_{s-1} \\&= \sum_{|I| \geq s \geq 2} \sum_{(\delta_1 ,.., \delta_s) \in \mathcal{C}(U,s)} \int_{0 \leq \sum_{i=1}^{s-1} t_i \leq t, 0 \leq t_i}  e^{-\beta_0 t  } K^{s-1}   dt_1 ,..., dt_{s-1} \\ &=  \sum_{|I| \geq s \geq 2} \sum_{(\delta_1 ,.., \delta_s) \in \mathcal{C}(U,s)} e^{-\beta_0 t  } K^{s-1} \frac{t^{s-1}}{(s-1)!}. \end{align*}
Take $x \in (0,1)$,  note that we have obtained: 
\[ \PP(\forall u \leq t; X_u \not \in \mathcal{V} \cup \{ \gamma^{\G_\rho} \}) \leq   \left(\frac{e^{-\beta_0}}{x}\right)^t \sum_{s \geq 2}^{|I|} K^{s-1} \sum_{(\delta_1 ,.., \delta_s) \in \mathcal{C}(U,s)} x^ t\frac{t^{s-1}}{(s-1)^1}. \]
Now, the function $\phi: \mathbb{R}_+ \rightarrow \mathbb{R}_+$ given by
\[ \phi(t) = x^t \frac{t^{s-1}}{(s-1)!}\]
is a positive function that vanishes at infinity. Then:
\[ C_1(x) = \max_{|I| \geq s \geq 2} \sup_{t \geq 0} K^{s-1} x^t \frac{t^{s-1}}{(s-1)!} < \infty. \]
So, we get
\[ \PP(\forall u \leq t: X_u \not \in \mathcal{V} \cup \{ \gamma^{\G_\rho} \} ) \leq \left( \frac{ e^{-\beta_0}}{x} \right)^t C(x) = \left( \frac{ e^{-\beta_0}}{x} \right)^t \sum_{s \leq 2}^{|I|}  \sum_{(\delta_1,..., \delta_s) \in \mathcal{C}(U,s)}  C_1(x), \]
where:
\[ C(x) = \sum_{s \geq 2}^{|I|} \sum_{(\delta_1 ,..., \delta_s) \in \mathcal{C}(U,s)} C_1 (x) < \infty. \]
By choosing $x$ such that $\frac{e^{-\beta_0}}{x} \leq e^{-\beta_0} + \theta $, and by defining $C(\theta) = C(x)$ we conclude that:
\[ \PP(\forall u \leq t: X_u \not \in \mathcal{V} \cup \{ \gamma^{\G_\rho} \} ) \leq C(\theta) (e^{-\beta_0} + \theta)^t. \]
\end{proof}
With the above lemmas we are able to prove the main result.
\begin{theorem}\label{QUASI1} We have $\eta > 0$ and $ \PP(\tau_{\V} < \infty)  > 0$. 
Also the exponential decay of $\PP(\tau >t)$ satisfies:
\[ \lim_{t \rightarrow \infty} e^{\eta t} \PP(\tau > t) = \lim_{n \rightarrow \infty} e^{ \eta t} \PP(\tau > t,  X_t \in \V) = \mathbb{E}(e^{\eta \tau_\V}, \tau_{\V} < \infty) \in (0,\infty). \]
The quasi-limiting distribution of $(X_t)_{t \geq 0}$ on $\Y^*(\G_\rho) \setminus \{ \gamma^{\G_\rho} \}$ is given by:
\begin{align}\label{qldf} \forall \delta \in \V:& \: \: \: \lim_{t \rightarrow \infty} \PP(X_t = \delta | \tau > t) = \frac{ \mathbb{E}(e^{ \eta \tau_\delta}, \tau_\delta < \infty)}{ \mathbb{E}(e^{\eta \tau_\V}, \tau_\V < \infty) }; \\ \nonumber \forall \delta \in \Y^*(\G) \setminus (\{ \gamma^{\G_\rho} \} \cup \V):& \: \: \: \: \lim_{t \rightarrow \infty} \PP(X_t = \delta | \tau > t) = 0. \end{align}
\end{theorem}
\begin{proof}
The result is shown in a very similar way as done on the main theorem of \cite{sm}. The fact that $\eta > 0$ and $ \PP(\tau_\V < \infty) > 0$ are proven on the same way.
\medskip

We claim that $\eta < \beta_0$. First, take $\delta \in \Delta \setminus \mathcal{V}$. By definition of $\eta$ we get
\[   \eta = \min\{ -Q_{\delta',\delta'}: \delta' \in \mathcal{V} \} < -Q_{\delta,\delta}. \]
Now take $\delta \not \in \Delta$. Then, given that $\gamma^{\G_\rho}$ is the absorbing state, there is a path on $\Y(\mathcal{G}_\rho)$; $\delta \lt \delta_1 \lt ... \lt \delta_r$ such $\delta_r \in \Delta$, and every $\delta_i$ is different. Then, $\PP_{\delta} (\tau_{\delta_r}  <\infty) > 0$. Given Lemma \ref{FSUM} $(i)$,
\[  Q_{\delta, \delta} < Q_{\delta_r,\delta_r} \leq -\eta \: \: \: \text{so} \: \: \:  \eta < - Q_{\delta,\delta}. \]
By taking the maximum on $\delta$ such that $\delta \in \Delta \setminus \V$ or $\delta \in \Delta$ we get $\eta < \beta_0$. This shows the claim.
Now we study:
\[ \PP( \tau > t) = \PP(\tau >t, X_t \not \in \mathcal{V}) + \PP(\tau > t, X_t \in \mathcal{V} ). \]
Since there is a path with positive probability from $\{I \}$ to every $\delta \in \Y(\mathcal{G}_\rho)$, $\delta \not = \{ I \}$ there for every $t_0 > 0$ we have:
\[ \forall \delta \in \mathcal{V}: \PP(\tau_\delta \leq t_0) > 0. \]
Consider, for a fixed $t_0>0$, $\alpha(\mathcal{V}) := min\{ \PP(\tau_\delta < t_0): \delta \in \mathcal{V} \} > 0$. From the Markov property, for every $\delta^* \in \mathcal{V}$ and $t > t_0$ we have:
\begin{align*} \PP(\tau > t) &\geq \int_{0}^{t_0} \PP(\tau > t | \tau_{\delta^*} = s) d \PP( \tau_{\delta^*} = s) \\ &= \int_{0}^{t_0} \PP_{\delta*} (\tau > t - s) d \PP(\tau_{\delta^*} = s) \\ &= \int_{0}^{t_0} e^{-\eta(t-s)} d \PP(\tau_{\delta^*} = s) \\ &\geq e^{-\eta t} \int_{0}^{t_0} d \PP(\tau_{\delta^*} = s) \\ &= e^{-\eta t} \PP(\tau_{\delta^*} < t_0) \geq e^{-\eta t} \alpha(\mathcal{V}), \end{align*}
then:
\[ \lim_{t \rightarrow \infty} \PP(\tau > t) e^{\eta t} \geq \alpha(\mathcal{V}) > 0. \]
This result together with Lemma \ref{decay}, in which we take $0 < \theta < e^{-\eta} - e^{-\beta_0}$, gives
\[ \PP(X_t \not \in \mathcal{V} | \tau > t) = \frac{ \PP( X_t \not \in \mathcal{V}, \tau > t  ) }{\PP(\tau > t)} \leq \frac{C}{\alpha(\mathcal{V})} \left(\frac{ e^{-\beta_0} + \theta }{ e^{-\eta} } \right)^t \xrightarrow[t \rightarrow 0]{} 0. \]
Then:
\[ \lim_{t \rightarrow \infty} \PP(X_t \in \mathcal{V} | \tau > t) = 1. \]
Now, let $\delta \in \mathcal{V}$. From the Markov property we get
\begin{align*} \PP( \tau > t, X_t = \delta) &= \int_{0}^{t} \PP( \tau > t, X_t = \delta | \tau_{\delta} = s) d \PP( \tau_{\delta} = s ) \\ &= \int_{0}^{t} \PP(\tau > t | \tau_\delta = s) d \PP(\tau_\delta = s)\\ &= \int_{0}^{t} \PP_\delta (\tau > t-s) d \PP(\tau_{\delta} = s) \\ &= \int_{0}^{t} \PP_\delta (X_{t-s} = \delta) d \PP(\tau_\delta = s ) \\ &= e^{-t \eta} \int_{0}^{t} e^{s \eta} d \PP(\tau_{\delta} = s). \end{align*}
We wish to have a control of the density $d \PP(\tau_\delta = s)$. By expanding forwardly the change of the Markov chain we obtain
\begin{align*} \PP(s < \tau_\delta \leq s + h) &= \sum_{\delta' \in \Y^*(\mathcal{G}_\rho)} \PP(s < \tau_\delta \leq s + h | X_s = \delta') \PP(X_s = \delta') \\ &= \sum_{\delta' \in \Y^*(\mathcal{G}_\rho)} \PP_{\delta'} (0< \tau_\delta \leq h) \PP(X_s = \delta')\\ &= o(h) + \sum_{\substack{\delta' \in \Y^*(\mathcal{G}_\rho):\\ \delta' \lt \delta, \delta' \not = \delta }} \PP_{\delta'}(\tau_\delta \leq h) \PP(X_s = \delta')\\ &= o(h) + \sum_{\substack{\delta' \in \Y^*(\mathcal{G}_\rho):\\ \delta' \lt \delta, \delta' \not = \delta  }} \PP_{\delta'}(X_h = \delta) \PP(X_s = \delta') \\ &= o(h) + \sum_{\substack{\delta' \in \Y^*(\mathcal{G}_\rho):\\ \delta' \lt \delta, \delta' \not = \delta }} (q_{\delta', \delta} h + o(h)) \PP(X_s = \delta'). \end{align*}
Put $M_\delta = \max\{ q_{\delta',\delta}: \delta' \lt \delta \}$, then we have proven:
\begin{align*}
\PP(s < \tau_\delta \leq s + h) &\leq o(h) + M_\delta h \mathbb{P}\left(\bigcup_{\substack{ \delta' \in \Y^*(\G_\rho): \\ \delta' \lt \delta, \delta' \not = \delta }} \{ X_s = \delta' \}\right) \\ & \leq o(h) + M_\delta h \PP(\forall u \leq s; X_u \not \in \V \cup \gamma^{\G_\rho}),
\end{align*} 
where in the last step we used the equality of Lemma \ref{FSUM} $(ii)$: $P_{\delta,\delta}^t  + P_{\delta,\gamma^{\G_\rho}}^t = 1$ for $\delta \in \mathcal{V}$. Hence choosing $\theta < e^{-\eta} - e^{-\beta_0}$ in Lemma \ref{decay} we obtain,
\[ \PP(s < \tau_\delta \leq s + h) \leq o(h) + M_\delta h (e^{-\beta_0}
+ \theta)^s. \]
Now, take $h > 0$,
\begin{align*}
\int_0^t e^{s \eta} d \mathbb{P}(\tau_\delta = s) &\leq \sum_{i=0}^{\lfloor \frac{t}{h} \rfloor - 1 } e^{h(i+1)\eta} \PP(i h < \tau_\delta \leq h(i+1)) \\ &\leq \sum_{i=0}^{\lfloor \frac{t}{h} \rfloor - 1 } e^{h(i+1)\eta} (o(h) + M_\delta h (e^{-\beta_0}+\theta)^{i h} \\ &= \frac{o(h)}{h}\sum_{i=0}^{\lfloor \frac{t}{h} \rfloor - 1 } e^{h(i+1)\eta} h + M_\delta h  e^{h \eta} \sum_{i=0}^{\lfloor \frac{t}{h} \rfloor - 1 }  \frac{(e^{-\beta_0}-\theta)^{i h}}{e^{i h \eta} } \\ & \leq \frac{o(h)}{h}\sum_{i=0}^{\lfloor \frac{t}{h} \rfloor - 1 } e^{h(i+1)\eta} h + M_\delta h  e^{h \eta} \frac{1}{1-\xi^h},
\end{align*}
where $\xi = \frac{(e^{-\beta_0}-\theta)}{e^{\eta}} \in (0,1)$. Now let $h \rightarrow 0^+$. With that we get
\[ \lim_{h \rightarrow 0^+} \sum_{i=0}^{\lfloor \frac{t}{h} \rfloor - 1 } e^{h(i+1)\eta} h = \int_{0}^{t} e^{s \eta} ds < \infty, \]
and so
\[ \lim_{h \rightarrow 0^+} \frac{o(h)}{h}\sum_{i=0}^{\lfloor \frac{t}{h} \rfloor - 1 } e^{h(i+1)\eta} h = 0. \]
Therefore, we have obtained the bound:
\begin{align*}
\int_0^t e^{s \eta} d \mathbb{P}(\tau_\delta = s) \leq \lim_{h \rightarrow 0} M_\delta e^{h \eta} \frac{h}{1-e^{h \log(\xi)}} = -M_\delta \frac{1}{\log(\xi)} < \infty.  \\ 
\end{align*}
This is uniform on $t$. So,
\begin{equation}\label{LIM1} \lim_{t \rightarrow \infty} e^{t \eta} \PP(\tau > t, X_t = \delta) = \int_{0}^{\infty} e^{s \eta} d \PP(\tau_\delta = s) = \mathbb{E}(e^{\eta \tau_\delta}, \tau_\delta < \infty) < \infty. \end{equation}
Given that by Lemma \ref{FSUM} $(ii)$, $P_{\delta,\delta}^t +  P_{\delta,\gamma^{\G_{\rho}}}^t = 1$ for all $\delta \in \mathcal{V}, t \geq 0$ it follows that $ \tau_{\delta} < \infty \Rightarrow \tau_\mathcal{V} = \tau_\delta$. Therefore:
\[ e^{ \eta \tau_\mathcal{V} } 1_{\tau_\mathcal{V} < \infty} = \sum_{\delta \in \mathcal{V}} e^{\eta \tau_\delta} 1_{\tau_\delta < \infty}. \]
So using equation (\ref{LIM1}):
\[ \mathbb{E}(e^{ \eta \tau_\mathcal{V}}, \tau_\mathcal{V} < \infty) = \sum_{\delta \in \mathcal{V}} \mathbb{E}(e^{\eta \tau_\delta}, \tau_\delta < \infty ) < \infty. \]
And so we get:
\begin{align} \label{LIM2}
\nonumber \lim_{t \rightarrow \infty} e^{ \eta t} \mathbb{P}(\tau > t, X_t \in \mathcal{V}) &= \lim_{t \rightarrow \infty}  \sum_{\delta \in \mathcal{V}} e^{\eta t} \mathbb{P}(\tau > t, X_t =\delta ) \\ &= \sum_{\delta \in \mathcal{V}} \mathbb{E}(e^{\eta \tau_\delta}, \tau_\delta < \infty ) = \mathbb{E}(e^{\eta \tau_\mathcal{V}}, \tau_\mathcal{V} < \infty). \end{align}
Finally identity (\ref{qldf}) follows directly from (\ref{LIM1}) and (\ref{LIM2}). This finishes the proof of the Theorem.
\end{proof}

Theorems \ref{RECSOL} and \ref{QUASI1} have a nice consequence for an approximation of solution of equation (\ref{EDOREC}). We recall the notation $o(e^{-t \eta})$ for a reminder that fulfills $\lim_{t \rightarrow \infty } \frac{|| o(e^{-t \eta})||}{e^{-t \eta}} = 0$.

\begin{theorem} \label{aproximation}
We have the following approximation for $\Xi_t \mu$:
\[ \Xi_t \mu = e^{-t \eta} \left[ (e^{t\eta}- \mathbb{E}(e^{\eta \tau_{\V}},\tau_\V < \infty))\bar{\mu} + (\sum_{\delta \in \V} \mathbb{E}(e^{\eta \tau_\delta}, \tau_\delta < \infty) \bigotimes_{L \in \delta} \mu_L)  \right] + o(e^{-t \eta}) \]
\end{theorem}
\begin{proof}
Note that, by Theorems \ref{RECSOL} and \ref{QUASI1}:
\begin{align*}
\Xi_t \mu &= \mathbb{E}(\mu^{(t)}) = \mathbb{E}(\mu^{(t)}, \tau \leq t) + \mathbb{E}(\mu^{(t)}, \tau > t) \\ &= \bar{\mu} \mathbb{P}(\tau \leq t) + \sum_{\delta \in \Y^*(\G_\rho) \setminus \{ \gamma^{\G_\rho}\}} \mathbb{P}(X_t = \delta , \tau > t)  \bigotimes_{L \in \delta} \mu_L \\ &= \bar{\mu} (1-e^{-\eta t} \mathbb{E}(e^{\eta \tau_\V}, \tau_\V < \infty)) + \sum_{\delta \in \Y^*(\G_\rho) \setminus \{ \gamma^{\G_\rho}\}} \mathbb{P}(X_t = \delta , \tau > t)  \bigotimes_{L \in \delta} \mu_L + o(e^{-t \eta}).
\end{align*}
So, since $||\cdot||$ fulfills the triangular inequality, we just need to prove that, for all $\hat{\delta} \not \in \V \cup \{ \gamma^{\G_\rho} \}$ and $\delta \in \V$:
\begin{align}
&\mathbb{P}(X_t = \hat{\delta}) e^{\eta t} \xrightarrow{t \rightarrow \infty} 0,  \label{tf1} \\
 & \frac{1}{e^{-t \eta}} \left| (\mathbb{E}(e^{\eta \tau_\delta}, \tau_\delta < \infty) e^{- t \eta}  - \mathbb{P}(X_t = \delta, \tau > t) ) \right| \xrightarrow{t \rightarrow \infty} 0. \label{tf2} 
\end{align}
Indeed, (\ref{tf1}) follows from Lemma \ref{decay}, and (\ref{tf2}) follows from (\ref{LIM1}).
\end{proof}
Finally, Theorem \ref{QUASI1} has the following consequences for the ratio limits and the $Q$-process, the latter is the Markov chain that avoids hitting the absorbing state. We will avoid the proof because it is entirely similar to Corollary 4.4 in \cite{sm}.
\begin{theorem} \label{QUASI2}

\begin{enumerate}[label={(\roman*)}]
\item For all $\delta \in \Y^*(\G) \setminus \{ \gamma^{\G_\rho} \}$ the following ratio is well defined:
\[ \lim_{t \rightarrow \infty} \frac{\PP_\delta (\tau > t)}{\PP(\tau > t)} = \frac{ \mathbb{E}_\delta (e^{\eta \tau_\V}, \tau_\V < \infty)}{\mathbb{E} (e^{\eta \tau_\V}, \tau_\V < \infty)},\]
and both expressions vanish when $\PP_\delta (\tau_{\V} < \infty) = 0$. 
\item For every $t > 0$ the vector $\varphi = (\varphi_\delta)_{\delta \in \Y^*(\G) \setminus \{ \gamma^{\G_\rho} \}} $ given by
\[ \varphi_\delta = \mathbb{E}_\delta (e^{\eta \tau_\V}, \tau_\V < \infty ), \]
is a right eigenvector of the restricted semi-group $(P^t)^* = P^t|_{\Y^*(\G) \setminus \{ \gamma^{\G_\rho} \}}$ with eigenvalue $e^{-\eta t}$.
\item For all $\{ \delta_i \}_{i=1}^{k} \subseteq \Y^*(\G) \setminus \{ \gamma^{\G_\rho} \}$ the following limit exists:
\[ \lim_{t \rightarrow \infty} \PP( X_{t_1}=\delta_1 , X_{t_2} = \delta_2,..., X_{t_k} = \delta_k \, | \, \tau > t ), \]
and defines a Markov process on $\partial_\V := \{ \delta: \PP_{\delta}(\tau_\V < \infty) >0 \}$, the states from which the process can arrive to $\V$, with generator:
\vspace{- 3 mm} \begin{align*} \widehat{Q}_{\delta, \delta'} &= Q_{\delta, \delta'} \frac{\varphi_{\delta'}}{\varphi_\delta} = Q_{\delta, \delta'} \frac{\mathbb{E}_\delta(e^{\eta \tau_\V}, \tau_\V < \infty)}{\mathbb{E}_{\delta'}(e^{\eta \tau_\V}, \tau_\V < \infty)} \: \: \: \delta, \delta' \in \partial_\V, \delta \not = \delta', \\
\widehat{Q}_{\delta,\delta} &= \eta + Q_{\delta,\delta} \: \: \: \delta \in \partial_\V.\end{align*}
\end{enumerate}
\end{theorem}
\begin{remark}
This result can be interpreted as follows; when conditioned to not hitting the absorbing state, the process arrives after a long time to a state on $\V$. Moreover one can compute the probability of arriving to some $\delta \in \V$ by using formula (\ref{qldf}). 
\end{remark}
\medskip
\noindent{\bf Acknowledgments}. We are thankful 
for support from the CMM Basal CONICYT Project AFB-170001. We are deeply grateful for discussions with Professor Ellen Baake, from University of Bielefeld, along all this work. Her comments contributed to improve significantly the final version of this manuscript.

\newpage 

\printbibliography

\smallskip

\noindent{ IAN LETTER, SERVET MART\'INEZ. } \\
\noindent{\it Departamento Ingenier{\'\i}a Matem\'atica and Centro
Modelamiento Matem\'atico, Universidad de Chile,
UMI 2807 CNRS, Casilla 170-3, Correo 3, Santiago, Chile.}
e-mail: iletter@dim.uchile.cl, smartine@dim.uchile.cl.

\end{document}